\documentclass[12pt,a4paper]{amsart}
\usepackage[top=35mm, bottom=35mm, left=30mm, right=30mm]{geometry}
\usepackage{mathptmx}
\usepackage{mathrsfs}
\usepackage{enumerate}
\theoremstyle{plain}
\newtheorem{thm}{Theorem}[section]
\newtheorem{exmp}[thm]{Example}
\newtheorem{lem}[thm]{Lemma}
\newtheorem{prop}[thm]{Proposition}
\newtheorem{cor}[thm]{Corollary}
\newtheorem{que}[thm]{Question}
\theoremstyle{definition}
\newtheorem{rem}[thm]{Remark}
\newtheorem{defn}[thm]{Definition}
\newcommand{\A}{\mathscr{A}}
\newcommand{\N}{\mathbb{N}}
\newcommand{\Z}{\mathbb{Z}}
\newcommand{\F}{\mathcal{F}}
\newcommand{\M}{\overline{\mathcal{M}}}
\newcommand{\ep}{\varepsilon}
\newcommand{\dn}{\Delta^{(n)}}
\newcommand{\htop}{h_{\textrm{top}}}
\newcommand{\B}{\mathscr{B}}
\newcommand{\set}[1]{\left\{#1\right\}}
\newcommand{\modn}[1]{\;(\text{mod }#1)}
\DeclareMathOperator{\diam}{diam}
\newcommand{\ra}{\rightarrow}
\DeclareMathOperator{\Prox}{Prox}
\DeclareMathOperator{\Rec}{Rec}
\DeclareMathOperator{\Tran}{Tran}
\DeclareMathOperator{\orbp}{Orb^+}
\begin{document}

\title{On $n$-scrambled tuples and distributional chaos in a sequence}
\author[J. Li]{Jian Li}
\date{\today}
\address[J. Li]{Department of Mathematics, Shantou University, Shantou, Guangdong, 515063, P.R. China
-- \and -- Department of Mathematics, University of Science and Technology of China,
Hefei, Anhui, 230026, P.R. China}
\email{lijian09@mail.ustc.edu.cn}

\author[P. Oprocha]{Piotr Oprocha}
\address[P.~Oprocha]{AGH University of Science and Technology, Faculty of Applied
Mathematics, al.
Mickiewicza 30, 30-059 Krak\'ow, Poland -- \and --  Institute of Mathematics\\ Polish Academy of Sciences\\
ul. \'Sniadeckich 8, 00-956 Warszawa, Poland} \email{oprocha@agh.edu.pl}
\begin{abstract}
The main aim of the present paper is to study relations between $n$-scrambled tuples and
their attraction-adherence properties with respect to various sequences of integers.
This extends previous research on relations between chaos in the sense of Li and Yorke
and distributional chaos with respect to a given sequence.
Moreover, we construct a system which is  $n$-distributionally chaotic but not $(n+1)$-chaotic.
\end{abstract}
\maketitle

\section{Introduction}

One of the most known definitions of chaos expressed in terms of dynamics of pairs emerged from
the paper by Li and Yorke \cite{LY75} more than 35 years ago.
However, it is not the only definition of this kind by far.
For example, Schweizer and Sm\'{\i}tal proposed in \cite{SS94} to
extend Li and Yorke approach by measuring lower and upper densities of the rate of proximality of pairs.
Other authors realized that some dynamical systems can be distinguished by looking on dynamics of tuples,
e.g. uniformly positive entropy of order $n$ \cite{HY06} and $n$-scrambled tuples \cite{X05}.
In particular, interval maps with zero topological entropy can have scrambled pairs,
but never contain scrambled triples \cite{L11}.
It can also be proved, that while there are many maps chaotic in the sense of Li and Yorke
which are not distributionally chaotic, there always exists
a sequence along which calculated lower and upper densities behave as in the definition of distributional chaos
(see \cite{O10} or \cite{LT10}).
Note that such a situation is quite natural in dynamics (e.g. compare the definition of entropy and sequence entropy).

In this paper we combine two above mentioned approaches, that is $n$-scrambled tuples and
densities of proximality, and  follow the ideas from \cite{TX09},
expressing chaos in terms of attaching and adherence of points defined with respect to Furstenberg families.
The paper is organized as follows. In Section 2, we introduce main concepts, terminology and notation.
In Section 3, some basic properties of the attraction and adherence of sets via Furstenberg families are discussed.
In Section 4, we investigate scrambled tuples and $n$-chaos via Furstenberg families.
We also provide a general condition for proving $n$-distributional chaos in a sequence.
In Section 5, we investigate uniformly chaos and show that every uniformly chaotic set is
$n$-distributionally chaotic in some sequence for every $n\geq 2$.
In Section 6, we restrict our considerations to interval maps.
For interval maps with topological entropy zero we provide a class of sequences
that may never lead to distributional chaos.
Finally, we construct an example showing that existence of an uncountable distributionally $n$-scrambled set
may be not enough for an uncountable $(n+1)$-scrambled set to exist
(however, our example still contains some $(n+1)$-scrambled tuples).

\section{Preliminaries}
Denote by $\N$ the set of all positive integers, and put $\N_0=\N\cup\{0\}$.
We will denote by $\B$ the set of all infinite subsets of $\N$. Since
every infinite strictly increasing sequence in $\N$ uniquely defines a set $Q\in \B$ and vice-versa, we will
not distinguish between sequences and sets in $\B$.

Recall that a closed Hausdorff space is \emph{perfect} if it has no isolated points,
and a \emph{Cantor  space} if it is a non-empty, compact, totally disconnected, perfect metrizable space.
We say that a subset in a Hausdorff space is a \emph{Cantor set} if it is a Cantor space with respect to
the relative topology, and a \emph{Mycielski set} if it can be presented as a countable
union of Cantor sets.

Throughout this paper the pair $(X,f)$ always denotes a \emph{(topological) dynamical system}
(or TDS for short), where $X$ is a non-empty compact metric space endowed with a metric $d$
and $f$ is a continuous map from $X$ into itself.
A non-empty closed invariant subset $Y\subset X$ defines naturally a \emph{subsystem}
$(Y,f)$ of $(X,f)$.

For $n\geq 2$, we denote by $(X^n, f^{(n)})$
the $n$-fold product system $(X\times X \times\cdots\times X, f\times f\times\cdots\times f)$, we put
$\Delta_n=\{(x,x,\ldots,x)\in X^n: x\in X\}$ and
$\Delta^{(n)}=\{(x_1,x_2,\ldots,x_n)\in X^n: x_i=x_j \text{ for some }i\neq j\}$. For $\delta>0$,
put $[A]_\delta=\{x\in X: \inf_{y\in A}d(x,y)<\delta\}$.

For $(X,f)$ and $n \geq 1$, we define the set of recurrent $n$-tuples and proximal $n$-tuples, by respectively,
\begin{align*}
\Rec_n(f)&=\{(x_1,\dotsc,x_n):\, \forall \varepsilon>0,\, \exists k\in\mathbb{N},\,
\textrm{s.t. } d(f^k(x_i),x_i)<\varepsilon, \forall i\}\\
\Prox_n(f)&=\{(x_1,\dotsc,x_n):\, \forall \varepsilon>0,\, \exists k\in\mathbb{N},\,
\textrm{s.t. } d(f^k(x_i),f^k(x_j))<\varepsilon, \forall i,j\}.
\end{align*}
For simplicity of notation, we write $Rec(f)=Rec_1(f)$ and $Prox(f)=Prox_2(f)$.
Clearly $Rec(f^{(n)})=Rec_n(f)$ and both sets $Rec_n(f)$ and $Prox_n(f)$ are $G_\delta$ subsets of $X^n$.

Recall that a system $(X,f)$ is  \emph{(topologically) transitive} if
for every two non-empty open subsets $U$ and $V$ of $X$ there exists an integer $n>0$
such that  $U\cap f^{-n}(V)\neq\emptyset$;
\emph{(topologically) weakly mixing} if $f\times f$ is transitive;
\emph{(topologically) exact} if for every non-empty open subset $U$ of $X$
there exists an integer $n>0$ such that $f^n(U)=X$.

By $\omega(x,f)$ we denote the $\omega$-limit set of $x$, that is the set of limits points of the positive orbit of $x$,
$\orbp(x,f)=\set{x,f(x),f^2(x),\ldots}$ treated as a sequence.
A point $x\in X$ is called a \emph{transitive point} if $\omega(x,f)=X$.
It is easy to see that if $(X,f)$ is transitive, then the set of all transitive points,
denoted by $\Tran(f)$, is a dense $G_\delta$ subset of $X$.

In \cite{LY75}, Li and Yorke initiated a study of dynamics of pairs as a tool in the description of
the complexity of a dynamical system.
This definition can be extended in a natural way to $n$-tuples, e.g. as done by Xiong in \cite{X05}.
Let us recall this definition here:

\begin{defn}
Let $(X,f)$ be a TDS, fix an integer $n\geq 2$ and $\delta>0$. We say that a tuple $(x_1,x_2,\ldots,x_n)\in X^n$ is \emph{$n$-$\delta$-scrambled}  if
$$
\liminf_{k\to\infty}\max_{1\leq i<j\leq n}d\left(f^k(x_i), f^k(x_j)\right)=0
$$
and
$$
\limsup_{k\to\infty}\min_{1\leq i<j\leq n}d\left(f^k(x_i), f^k(x_j)\right)>\delta.
$$
\end{defn}

\begin{defn}
Given an integer $n\geq 2$ and $\delta>0$, a subset $C$ of $X$ is called \emph{$n$-$\delta$-scrambled} if every tuple $(x_1,x_2,\ldots,x_n)\in
C^n\setminus\Delta^{(n)}$ is $n$-$\delta$-scrambled.

If the above condition holds only with $\delta=0$, then we say that $C$ is \emph{$n$-scrambled}.
We say that a dynamical system $(X,f)$ is \emph{Li-Yorke $n$-chaotic} (resp.  \emph{Li-Yorke $\delta$-$n$-chaotic}),
if there exists an uncountable $n$-scrambled set (resp.  $\delta$-$n$-scrambled set).
\end{defn}

It is proved in \cite{X05} that if a non-periodic transitive system has a fixed point
then it is Li-Yorke $n$-chaotic for every $n\geq 2$.
Recently, it was proved in \cite{L11} that every zero entropy Li-Yorke $2$-chaotic interval map
does not contain any $3$-scrambled tuple.

In \cite{SS94}, Schweizer and Sm\'{\i}tal extended  the approach of Li and Yorke, introducing another kind of chaos,
which is presently called \emph{distributional chaos}.
In a natural the definition of Schweizer and Sm\'{\i}tal can be extended from pairs to tuples.
In this definition we can also calculate density functions over subsequences of iterates instead of $\N$
(e.g.  see \cite{Wang07}).
Let us state these extended definitions in a more formal way.

Let $(X,f)$ be a TDS and let $Q=\{q_k\}_{k=1}^{\infty}\in \B$.
For $x_1,x_2,\ldots,x_n\in X$, $t>0$ and $n\geq 2$, put
\[\Phi_{(x_1,x_2,\ldots,x_n)}(t|Q)=\liminf_{m\to\infty}\frac{1}{m}
\#\left\{1\le k\le m: \min_{1\leq i<j\leq n}d\left(f^{q_k}(x_i), f^{q_k}(x_j)\right) < t\right\}\]
and
\[\Phi^*_{(x_1,x_2,\ldots,x_n)}(t|Q)=\limsup_{m\to\infty}\frac{1}{m}
\#\left\{1\le k\le m: \max_{1\leq i<j\leq n}d\left(f^{q_k}(x_i), f^{q_k}(x_j)\right) < t\right\}.\]

\begin{defn}
Let $(X,f)$ be a TDS and $Q\in\B$.
A tuple $(x_1,x_2,\ldots,x_n)\in X^n$ is \emph{distributionally $n$-$\delta$-scrambled in the sequence $Q$} if
\begin{enumerate}[(1)]
\item $\Phi^*_{(x_1,x_2,\ldots,x_n)}(t|Q)=1$ for every $t>0$, and \\
\item $\Phi_{(x_1,x_2,\ldots,x_n)}(\delta|Q)=0$.
\end{enumerate}

A subset $D$ of $X$ is called \emph{distributionally $n$-$\delta$-scrambled in the sequence $Q$}
if every tuple $(x_1,x_2,\ldots,x_n)\in D^n\setminus\dn$ is distributionally $n$-$\delta$-scrambled in the sequence $Q$.
The system $(X,f)$ is called \emph{distributionally $n$-$\delta$-chaotic in the sequence $Q$}
if there exists an uncountable distributionally $n$-scrambled set in the sequence $Q$.

Similarly, we can define \emph{distributionally $n$-scrambled sets in the sequence $Q$}
and \emph{distributionally $n$-chaotic systems in the sequence $Q$}.
\end{defn}

When $Q=\N$, for simplicity we omit $Q$ in the above notation,
i.e. we write $\Phi_{(x_1,x_2,\ldots,x_n)}(t)$ instead of $\Phi_{(x_1,x_2,\ldots,x_n)}(t|\N)$,
distributionally $n$-scrambled instead of distributionally $n$-scrambled in the sequence $\N$, etc.

\section{Attraction, Adherence and relative densities of sets}
Recently many relations between Furstenberg families and properties of dynamical systems
were obtained by various authors, e.g. see \cite{A97,O10,TX09} and \cite{XLT07}.
Recall that a \emph{Furstenberg family} (or simply a \emph{family}) $\F$ is a collection of subsets of $\N$
which is \emph{upwards hereditary},
that is
\[F_1 \in \F \text{ and } F_1 \subset F_2 \quad \Longrightarrow \quad F_2 \in \F.\]
In this section,
we consider the attraction and adherence of sets via families defined by relative density.

Let $Q=\{q_k\}_{k=1}^\infty\in\B$ and $P\subset \N$. The \emph{upper density of $P$ with respect to $Q$}
is defined by
\[ \overline{d}(P\mid Q)=\limsup_{m\rightarrow\infty}\frac{\#(P\cap\{q_1,\cdots,q_m\})}{m},\]
where as usual $\#(A)$ denotes the cardinality of a set $A$.
For every $a\in [0,1]$, we define
\[\overline{\mathcal M}_Q(a)=\big\{P\subset \N: P\cap Q \textrm{ is infinite and }\overline{d}(P\mid Q)\geq a\big\}.\]
Clearly, $\overline{\mathcal M}_Q(a)$ is a Furstenberg family.

The following lemma comes from \cite{GX} where its utility
to dynamical systems was first presented. The paper \cite{GX} is hardly available,
but the careful reader should be able to prove this lemma by himself.

\begin{lem}\label{lem:SEQ}
Let $\{S_i\}_{i=1}^{\infty}$ be a sequence in $\B$.
Then there exists $Q\in \B$
such that $\overline{d}(S_i\mid Q)=1$ for every $i=1,2,\ldots$.
\end{lem}

Let $(X,f)$ be a TDS, $x\in X$ and $A$ be a subset of $X$ and let $\F$ be a family. We denote
$N(x,A)=\{n\in\N: f^n(x)\in A\}$
and say that a point $x\in X$ is:
\begin{enumerate}
\item  an \emph{$\F$-attaching point of $A$} if $N(x, A)\in \F$,
\item  an \emph{$\F$-adherent point of $A$} if for every $\ep>0$,
$x$ is an $\F$-attaching point of the set $[A]_\delta$.
\item  an \emph{$\F$-$\delta$-escaping point of $A$}
if $x$ is an $\F$-attaching point of the set $X\setminus \overline{[A]_\delta}$.
\end{enumerate}

Denote by $\overline{\mathcal M}_Q(a,A)$ the set of $\overline{\mathcal M}_Q(a)$-attaching points of $A$.
If we want to emphasize that the map $f$ is acting on $X$, we use the notation  $\overline{\mathcal M}_Q(a,A,f)$.

The following Lemma extends Theorem 3.2 of \cite{XLT07}.
\begin{lem}
Let $(X,f)$ be a TDS, $Q\in\B$ and $a\in [0,1]$ and let $A$ be a non-empty open subset of $X$.
Then $\overline{\mathcal M}_Q(a,A)$ is a $G_\delta$ subset of $X$.
\end{lem}
\begin{proof}
Let $Q=\{q_k\}_{k=1}^\infty$. For $m\in\N$, put
\[g_m: X\to [0,1],\quad x\mapsto \frac{\#(\{1\leq k\leq m: f^{q_k}(x)\in A\})}{m}\]
and
\[g: X\to[0,1],\quad x\mapsto \overline{d}(N(x, A)\mid Q).\]
Clearly, $g(x)=\limsup_{m\to\infty}g_m(x)$.
Since $A$ is open, it is easy to see that each $g_m$ is continuous,
then $g$ is lower semi-continuous.
By the definition of semi-continuity, we have
$\{x\in X: g(x)> a-1/r\}$ is open for every $r\in\N$ and therefore
\[ \overline{\mathcal M}_Q(a,A)=\big\{x\in X: g(x)\geq a\big\}
=\bigcap_{r=1}^\infty\left\{x\in X: g(x)> a-\frac{1}{r}\right\}\]
is a $G_\delta$ subset of $X$.
\end{proof}

By the definition of adherence and the above lemma, we have
\begin{cor}\label{cor:G-delta}
Let $(X,f)$ be a TDS,  $Q\in\B$, $a\in [0,1]$ and $\delta>0$. Suppose that  $A$ is a subset of $X$.
Then
\begin{enumerate}
\item the set of $\overline{\mathcal M}_Q(a)$-adherent points of $A$ is a $G_\delta$ subset of $X$.
\item the set of $\overline{\mathcal M}_Q(a)$-$\delta$-escaping points of $A$ is a $G_\delta$ subset of $X$.
\end{enumerate}
\end{cor}

Recall that the \emph{floor} and \emph{ceiling} functions on real numbers are defined as, respectively,
\[\lfloor x \rfloor=\max\, \{m\in\mathbb{Z}\mid m\le x\} \ \textrm{and}\
\lceil x \rceil=\min\,\{n\in\mathbb{Z}\mid n\ge x\}.\]
For every $r\in\N$, define
\[\varphi_r:\N\to \N_0,\ n\mapsto \left\lfloor \frac{n}{r}\right\rfloor
\ \textrm{and}\
\phi_r: \N\to \N,\  n\mapsto \left\lceil \frac{n}{r}\right\rceil.\]
Observe that both functions $\varphi_r$ and $\phi_r$ are $r$-to-$1$.

\begin{lem} \label{lem:f-and-f-r}
Let $Q\in\B$, $\delta>0$ and $r\in \N$ and assume that $A$ is an $f$-invariant closed subset of $X$.
\begin{enumerate}[(1)]
\item\label{lem:f-and-f-r:c1} Let $P=\phi_r(Q)$. If $x$ is an $\overline{\mathcal M}_Q(a)$-adherent point of $A$ for $f$,
then $x$ is also an $\overline{\mathcal M}_P(a)$-adherent point of $A$ for $f^r$.
\item\label{lem:f-and-f-r:c2} Let $P=\varphi_r^{-1}(Q)$. If $x$ is an $\overline{\mathcal M}_Q(a)$-adherent point of $A$ for $f^r$,
then $x$ is also an $\overline{\mathcal M}_P(a)$-adherent point of $A$ for $f$.
\item\label{lem:f-and-f-r:c3} Let $P=\varphi_r(Q)$. If $x\in\overline{\mathcal M}_Q(a,X\setminus \overline{[A]_\delta},f)$,
then there exists $\delta'=\delta'(f,\delta,A)>0$ such that
$x\in \overline{\mathcal M}_P(a,X\setminus \overline{[A]_{\delta'}},f^r)$.
\item\label{lem:f-and-f-r:c4} Let $P=\phi_r^{-1}(Q)$. If $x\in \overline{\mathcal M}_Q(a,X\setminus \overline{[A]_\delta},f^r)$,
then there exists $\delta'=\delta'(f,\delta,A)>0$
such that $x\in \overline{\mathcal M}_P(a,X\setminus \overline{[A]_{\delta'}},f)$.
\end{enumerate}
\end{lem}

\begin{proof}
Let $Q=\{q_k\}_{k=1}^\infty$.
For every $\ep>0$, by the continuity of $f$ and the compactness of $X$,
there exists $\ep'>0$ such that if $d(y,A)<\ep'$ then $d(f^i(y),A)<\ep$ for $i=0,1,\ldots,r-1$.
As a consequence, if we put $\delta=2\ep$ and $\delta'=\ep'/2$ then
$d(f^{m-i}(y),A)>\delta'$ for $i=0,1,\ldots,r-1$ provided that $d(f^m(y),A)>\delta$ and $m\geq r$.

First we prove \eqref{lem:f-and-f-r:c1}. Note that if $f^n(x)\in [A]_{\ep'}$
then $f^n(x), f^{n+1}(x),\ldots,f^{n+r-1}(x)\in [A]_\ep$, therefore
if $n \in N(x,[A]_{\ep'},f)$ then $\phi_r(n)\in N(x,[A]_{\ep},f^r)$.

Let $P=\phi_r(Q)$ and arrange $P$ as $\{p_k\}_{k=1}^\infty$.
Since $\phi$ is $r$-to-$1$, we have
\[\#\big(N(x,[A]_{\ep'},f)\cap \{q_1,\ldots,q_k\}\big)\leq r
\#\big(N(x,[A]_{\ep},f^r)\cap \{p_1,\ldots,p_{\phi_r(k)}\}\big)\] and
thus $\overline{d}(N(x,[A]_{\ep'},f)|Q)\leq \overline{d}(N(x,[A]_{\ep},f^r)|P)$.

If $x$ is an $\overline{\mathcal M}_Q(a)$-adherent point of $A$ for $f$,
that is for every $\ep'>0$, $\overline{d}(N(x,[A]_{\ep'},f)|Q)\geq a$, then
for every $\ep>0$ we have $\overline{d}(N(x,[A]_{\ep},f^r)|P)\geq a$.
This shows that $x$ is also an $\overline{\mathcal M}_P(a)$-adherent point of $A$ for $f^r$.

For the proof of \eqref{lem:f-and-f-r:c2}
observe that if $n\in N(x,[A]_{\ep'},f^r)$, then ${\varphi_r}^{-1}(n)\subset N(x,[A]_\ep,f)$.
Let $P=\varphi_r{}^{-1}(Q)$ and arrange $P$ as $\{p_k\}_{k=1}^\infty$.
Since $\varphi$ is $r$-to-$1$, we have
\[r \#\big(N(x,[A]_{\ep'},f^r)\cap \{q_1,\ldots,q_{k}\}\big)\leq
\#\big(N(x,[A]_{\ep},f)\cap \{p_1,\ldots,p_{rk}\}\big).\]
Hence $\overline{d}(N(x,[A]_{\ep'},f^r)|Q)\leq \overline{d}(N(x,[A]_{\ep},f)|P)$.
Now \eqref{lem:f-and-f-r:c2} follows by the same arguments as \eqref{lem:f-and-f-r:c1}.

Proofs of \eqref{lem:f-and-f-r:c3} and \eqref{lem:f-and-f-r:c4} have analogous proofs to \eqref{lem:f-and-f-r:c1} and \eqref{lem:f-and-f-r:c2},
thus we leave them to the reader.
\end{proof}

\section{Scrambled tuples and $n$-scrambled sets}

In this section, following the ideas contained in \cite{TX09},
we will use the notions of adherence and escaping to define scrambled tuples.

\begin{defn}
Let $(X,f)$ be a TDS, $\F_1$ and $\F_2$ be two families, $\delta>0$ and $n\geq 2$. We say that
a tuple $(x_1,x_2,\ldots,x_n)\in X^n$ is \emph{$(\F_1,\F_2)$-$n$-$\delta$-scrambled}
if $(x_1,x_2,\ldots, x_n)$ is an $\F_1$-adherent point of $\Delta_n$ and
an $\F_2$-$\delta$-escaping point of $\dn$.

A subset $C$ of $X$ is called \emph{$(\F_1,\F_2)$-$n$-$\delta$-scrambled} if every tuple $(x_1,x_2,\ldots,x_n)\in
C^n\setminus\Delta^{(n)}$ is $(\F_1,\F_2)$-$\delta$-$n$-scrambled for some $\delta>0$.

The system $(X,f)$ is called \emph{$(\F_1,\F_2)$-$\delta$-$n$-chaotic}
if there exists an uncountable $(\F_1,\F_2)$-$\delta$-$n$-scrambled set.

Similarly, we can define $(\F_1,\F_2)$-$n$-scrambled sets  and $(\F_1,\F_2)$-$n$-chaotic systems.
\end{defn}

The following fact is almost immediate consequence of the definitions.
\begin{prop}\label{prop:scram-eq}
A tuple
$(x_1,x_2,\ldots, x_n)\in X^n$ is $n$-scrambled
(resp.  distributionally $n$-scrambled in a sequence $Q$)
if and only if
it is $(\overline{\mathcal M}(0),\overline{\mathcal M}(0))$-$n$-scrambled
(resp.  $(\overline{\mathcal M}_Q(1),\overline{\mathcal M}_Q(1))$-$n$-scrambled).
\end{prop}

For convenience we restate here a
version of Mycielski's theorem (\cite[Theorem~1]{My64}) which we shall use.

\begin{thm}[Mycielski Theorem]\label{thm:Mycielski-thm}
Let $X$ be a complete second countable metric space without isolated points.
If $R$ is a dense $G_\delta$ subset of $X^n$, then there exists a
dense Mycielski set $K\subset X$ such that $K^n\setminus\dn\subset R$.
\end{thm}

It is still an open problem whether a Cantor scrambled set can be selected when an uncountable scrambled set
exists in the system. The positive answer is only known in the case of $\delta$-scrambled sets (e.g. see \cite{BHS08}).
However in many cases when it is possible to construct an uncountable scrambled set, there also exists a
function measuring separation of scrambled pairs. In most of these cases Cantor scrambled set can also be constructed.
This observation is the main motivation behind the next theorem.

\begin{thm}\label{thm:eta}
Let $(X,f)$ be a TDS and $D$ be an infinite subset of $X$ such that
its closure $Y=\overline{D}$ has no isolated points and $D^n \subset \Prox_n(f)$.
Assume that there is a continuous function $\eta\colon Y^n\to[0,\infty)$ with
$\eta(Y^n\setminus\dn)\subset(0,\infty)$, and
a dense subset $C$ of $Y^n\setminus\dn$ such that for any $n$-tuple $(x_1,\dotsc,x_n)\in C$ we have
\[ \limsup_{m\to \infty} \min_{1\leq i<j\leq n}(f^m(x_i),f^m(x_j))>\eta(x_1,\dotsc,x_n). \]
Then there exists $Q\in\B$ and
a dense Mycielski $(\overline{\mathcal M}_Q(1),\overline{\mathcal M}_Q(1))$-$n$-scrambled subset $K$ of $Y$.
Moreover, each tuple $(x_1,\dotsc,x_n)\in K^n\setminus\dn$ is an
$\overline{\mathcal M}_Q(1)$-$\delta$-escaping point of $\dn$,
where $\delta=\frac 1 2 \eta(x_1,\dotsc,x_n)$.
\end{thm}

\begin{proof}
Since $X$ is separable, without loss of generality, taking subsets if necessary,
we may assume that both $C$ and $D$ are countable, while still $\overline{D}=Y$ and $\overline{C}=Y^n$.
For every $(x_1,\dotsc,x_n)\in C$, there exists $S_{(x_1,\dotsc,x_n)}=\{S_{(x_1,\dotsc,x_n)}(k)\}_{k=1}^\infty\in\B$
such that
\[
\lim_{k\to\infty}\min_{1\leq i<j\leq n}d(f^{S_{(x_1,\dotsc,x_n)}(k)}(x_i), f^{S_{(x_1,\dotsc,x_n)}(k)}(x_j))
>\eta(x_1,\dotsc,x_n).
\]
Similarly, for every $(y_1,\dotsc,y_n)\in D^n$,
there exists $P_{(y_1,\dotsc,y_n)}=\{P_{(y_1,\dotsc,y_n)}(k)\}_{k=1}^\infty\in\B$ such that
\[
\lim_{k\to\infty}\max_{1\leq i<j\leq n}d(f^{P_{(y_1,\dotsc,y_n)}(k)}(y_i), f^{P_{(y_1,\dotsc,y_n)}(k)}(y_j))=0.
\]
Then by Lemma \ref{lem:SEQ} we can find $Q=\{q_k\}_{k=1}^\infty\in\B$ such that
$$\overline{d}(P_{(y_1,\ldots, y_n)}|Q)=\overline{d}(S_{(x_1,\ldots, x_n)}|Q)=1$$
for every $(y_1,\dotsc,y_n)\in D^n$ and $(x_1,\dotsc,x_n)\in C$.

For every $m\in\N$, define
\begin{align*}
  A_m=&\Bigg\{ (x_1,\dotsc,x_n)\in Y^n\colon
  \#\bigg(\Big\{1\leq k\leq N\colon \min_{1\leq i<j\leq n}d(f^{q_k}(x_i),f^{q_k}(x_j))\\
  & \qquad >\eta(x_1,\dotsc,x_n)+\frac 1 M \Big\}\bigg)
  \geq N\Big(1-\frac 1m\Big) \textrm{ for some } N,M\geq m\Bigg\}.
\end{align*}
Observe that by the continuity of $f$ and $\eta$ each $A_m$ is open in $Y^n$.
By the choice of $Q$, we have $C\subset A_m$ and thus $A_m$ is open and dense in $Y^n$.
Therefore $A=\bigcap_{m=1}^\infty A_m$ is a residual subset of $Y^n$.
Moreover, by the construction, every tuple  $(x_1,\dotsc,x_n)\in A$ is  a
$\overline{\mathcal M}_Q(1)$-$\delta$-escaping point of $\dn$,
where $\delta=\frac 1 2 \eta(x_1,\dotsc,x_n)$.

Similarly, if we denote by $R\subset Y^n$ the set of all $\overline{\mathcal M}_Q(1)$-adherent points of $\Delta_n$ then
$D^n\setminus\dn\subset R$ which by Corollary~\ref{cor:G-delta} implies that $R$ is residual in $Y^n$.
Then by Theorem~\ref{thm:Mycielski-thm} there exists a dense
Mycielski subset $K$ of $Y$ such that $K^n\setminus\dn\subset A\cap R$.
Clearly, $K$ is $(\overline{\mathcal M}_Q(1),\overline{\mathcal M}_Q(1))$-$n$-scrambled. This ends the proof.
\end{proof}

\begin{rem}
There are a few standard situations when the existence of scrambled sets can be proved.
In these situations function $\eta$ and sets $C,D$ can be easily defined, which show the utility of Theorem~\ref{thm:eta}
as a unified tool for proving chaos in the sense of Li and Yorke, as well as existence of Cantor scrambled sets.
\end{rem}

\begin{exmp} Here we present a few situations when Theorem~\ref{thm:eta} can be applied.
\begin{enumerate}[1.]
\item If $D$ is an uncountable $\delta$-$n$-scrambled set,
then the closure $\overline{D}$ is the union of a perfect set and an at most countable set,
so without loss of generality we may assume that $\overline{D}$ is perfect.
Then it is enough to put $C=D$ and $\eta\equiv \delta$ to apply Theorem \ref{thm:eta}. Therefore
there exists $Q\in\B$ and a dense Mycielski
$(\overline{\mathcal M}_Q(1),\overline{\mathcal M}_Q(1))$-$n$-$\frac{\delta}{2}$-scrambled subset
$K$ of $\overline{D}$.
In particular, if $D$ is dense in $X$ then $K$ is also dense in $X$.
This generalizes results of \cite{O10} and \cite{LT10}.

\item  If $(X,f)$ is weakly mixing, then it is easy to verify that $\overline{\Tran(f^{(n)})}=X^n$.
If $\#(X)>1$ then $X$ is perfect. If we put $\eta\equiv \ep$,
where $\ep=\frac{1}{2} \max\{ \min_{1\leq i<j\leq n}d(x_i,x_j):\,(x_1,\dotsc,x_n)\in X^n\}$,
select any countable subset $D\subset \Tran(f^{(n)})$ with $\overline{D}=X^n$ and put $C=D$
then assumptions of Theorem~\ref{thm:eta} are satisfied.
Thus there exists $Q\in\B$ and a dense Mycielski
$(\overline{\mathcal M}_Q(1),\overline{\mathcal M}_Q(1))$-$n$-$\ep$-scrambled subset of $X$.
The existence of scrambled sets for weakly mixing systems was proved independently in \cite{Iwanik}
and \cite{XY91}.

\item
If $(X,f)$ is transitive with a fixed point, then it is not hard to verify that $\overline{\Prox_n(f)}=\overline{\Rec_n(f)}=X^n$.
Again, if $\#(X)>1$, then $X$ is perfect.
So it is enough to put $\eta(x_1,\ldots, x_n)=\min_{i\neq j} d(x_i,x_j)$,
and select any $C\subset \Prox_n(f)$ and $D\subset \Rec_n(f)$ with both $C$ and $D$ dense in $X^n$.
Then by Theorem~\ref{thm:eta},
there exists $Q\in\B$ and a dense Mycielski
$(\overline{\mathcal M}_Q(1),\overline{\mathcal M}_Q(1))$-$n$-scrambled subset of $X$.
The existence of scrambled sets in this context was first proved in \cite{HY02} and extended in \cite{X05}.
\end{enumerate}
\end{exmp}

Consider the set $\A=\{0, 1, \dots,n-1\}$, $n\geq 2$ endowed with discrete topology and
let $\Sigma_n^+=\A^{\N_0}$ be the product of infinitely many copies of
$\A$ with the product topology.
The shift transformation is a continuous map $\sigma\colon \Sigma_n^+ \to \Sigma_n^+$ given by
$\sigma(x)_i = x_{i+1}$.
The dynamical system $(\Sigma_n^+,\sigma)$ is called the (one-sided) \emph{full shift (on $n$ symbols)}.
Any closed subset $X\subset \Sigma_n^+$
invariant for $\sigma$ is called a \emph{subshift} of $\Sigma_n^+$.
We can also consider the \emph{two-sided full shift} $\Sigma_n=\A^{\mathbb{Z}}$,
which again is a metrizable compact space with the continuous shift
 transformation $\sigma$ defined by the same formula as before (but now $i\in \Z$).

\begin{thm}\label{thm:exact}
Let $(X,f)$ be a topologically exact TDS with $\#(X)>1$.
Then for every $n\geq 2$, there exists $\delta>0$ such that for every $Q,P\in\B$ there exists
a dense Mycielski $(\overline{\mathcal M}_Q(1),\overline{\mathcal M}_P(1))$-$n$-$\delta$-scrambled set.
\end{thm}
\begin{proof}
It is easy to see that there is $r>0$ and an invariant subset $\Lambda$ for $f^r$
such that $f^r|_\Lambda$ is an extension of $\Sigma_2^+$.
It is well known that $\Sigma_2^+$ contains an uncountable family of disjoint minimal subsystems,
therefore $\Lambda$ also contains uncountable family of minimal systems of $f^r$.
In particular, there are $n$-pairwise disjoint minimal subsets $A_1,A_2,\dots, A_n$ for $f$.
Denote  $B_j=\bigcup_{i=1}^\infty f^{-i}(A_j)$ and
observe that $B_j$ is dense in $X$ for $j=1,2,\ldots,n$.
If we put $\delta=\frac{1}{2}\min_{1\leq i<j\leq n}d(A_i,A_j)$
then for every $P\in \B$ taking any tuple in $B_1\times B_2\times\cdots\times B_n$ we always obtain an
$\overline{\mathcal M}_P(1)$-$\delta$-escaping point of $\dn$.
Thus the set of such tuples is residual in $X^n$ by Corollary~\ref{cor:G-delta}.

Similarly, if we fix $n$ non-empty open subsets $U_1,\ldots, U_n$ of $X$ and a point $z\in X$
then there are $x_i\in U_i$ for $i=1,2,\dotsc,n$ and $r>0$
such that $f^r(x_i)=z$ for $i=1,2,\ldots,n$.
In particular the tuple $(x_1,\ldots, x_n)$ is an $\overline{\mathcal M}_Q(1)$-adherent point of $\Delta_n$
for any $Q\in \B$.
Then as before, the set of $\overline{\mathcal M}_Q(1)$-adherent points of $\Delta_n$ is residual in $X^n$.

By the definition of $(\F_1,\F_2)$-scrambled tuple we immediately see that
the set of $(\overline{\mathcal M}_Q(1),\overline{\mathcal M}_P(1))$-$n$-$\delta$-scrambled tuples
is residual in $X^n$. Hence by  Mycielski Theorem
there is a dense Mycielski $(\overline{\mathcal M}_Q(1),\overline{\mathcal M}_P(1))$-$n$-$\delta$-scrambled set.
\end{proof}

The following proposition easily follows from Lemma \ref{lem:f-and-f-r} and the definitions.
\begin{prop}\label{prop:QP-n-r}
Let $(X,f)$ be a TDS, $Q,P\in\B$, $a,b\in[0,1]$, $\delta>0$, $n\geq 2$ and $r\in\N$.
\begin{enumerate}
\item There exists $\delta'>0$ such that if $(x_1,x_2,\dotsc,x_n)\in X^n$ is
$(\overline{\mathcal M}_Q(a),\overline{\mathcal M}_P(b))$-$n$-$\delta$-scrambled for $f$,
then $(x_1,x_2,\dotsc,x_n)$ is
$(\overline{\mathcal M}_{Q'}(a),\overline{\mathcal M}_{P'}(b))$-$n$-$\delta'$-scrambled for $f^r$,
where $Q'=\phi_r(Q)$ and $P'=\varphi_r(P)$.
\item There exists $\delta'>0$ such that  if $(x_1,x_2,\dotsc,x_n)\in X^n$ is
$(\overline{\mathcal M}_Q(a),\overline{\mathcal M}_P(b))$-$n$-$\delta$-scrambled for $f^r$,
then $(x_1,x_2,\dotsc,x_n)$ is
$(\overline{\mathcal M}_{Q'}(a),\overline{\mathcal M}_{P'}(b))$-$n$-$\delta'$-scrambled for $f$,
where $Q'=\varphi_r^{-1}(Q)$ and $P'=\phi_r^{-1}(P)$.
\end{enumerate}
\end{prop}

\begin{rem}
For every $A\in\B$, adding or removing from $A$ a finite number of elements does not affect its density.
Therefore in Proposition \ref{prop:QP-n-r} for $Q=P=\N$ we can also have $Q'=P'=\N$.
\end{rem}

\begin{cor}
Let $(X,f)$ be a TDS, $a,b\in[0,1]$, $n\geq 2$ and $r\in\N$.
Then $(X,f)$ is $(\overline{\mathcal M}(a),\overline{\mathcal M}(b))$-$n$-$\delta$-chaotic for some $\delta>0$
if and only if $(X,f^r)$ is $(\overline{\mathcal M}(a),\overline{\mathcal M}(b))$-$n$-$\delta'$-chaotic for some $\delta'>0$
\end{cor}

Note that there are numerous examples of $P,Q\in \B$ that cannot be obtained as inverse images
$Q=\varphi_r^{-1}(Q')$ or $P=\phi_r^{-1}(P')$ for any $P',Q'\in \B$.
However, we can still prove the following theorem.

\begin{thm}\label{cor:f-and-fr}
Fix any two integers $n,r\geq 2$. The following conditions are equivalent:
\begin{enumerate}
\item\label{cor:f-and-fr:c1} There is $\delta>0$ such that for every $Q,P\in\B$, there exists
a dense Mycielski $(\overline{\mathcal M}_Q(1),\overline{\mathcal M}_P(1))$-$n$-$\delta$-scrambled set
for $f$.

\item\label{cor:f-and-fr:c2}  There is $\delta>0$ such that for every $Q,P\in\B$, there exists
a dense Mycielski $(\overline{\mathcal M}_Q(1),\overline{\mathcal M}_P(1))$-$n$-$\delta$-scrambled set
for $f^r$.
\end{enumerate}
\end{thm}
\begin{proof}
$\eqref{cor:f-and-fr:c1} \Longrightarrow \eqref{cor:f-and-fr:c2}$  is an immediate consequence of Proposition~\ref{prop:QP-n-r}, so
we only need to show that $\eqref{cor:f-and-fr:c2}\Longrightarrow\eqref{cor:f-and-fr:c1}$.

Let $Q=\{q_k\}_{k=1}^\infty$.
For every $m\geq 1$, let $I_m=[mr, (m+1)r)$.
Put $Q'=\{m\in\N:\, Q\cap I_m\neq\emptyset\}$ and present $Q'$ as an increasing sequence $\{q'_k\}_{k=1}^\infty$.
First we are going to prove the following:

\textbf{Claim 1}: If $(x_1,\dotsc,x_n)\in X^n$ is an $\overline{\mathcal M}_{Q'}(1)$-adherent point of
$\Delta_n$ for $f^r$, then it is also an $\overline{\mathcal M}_{Q}(1)$-adherent point of
$\Delta_n$ for $f$.

\begin{proof}[Proof of the Claim 1]
For every $\ep>0$, by the continuity of $f$ and the compactness of $X$,
there exists $\ep'>0$ such that if $d(x,y)<\ep'$ then $d(f^i(x),f^i(y))<\ep$ for $i=0,1,\ldots,r-1$.
Then
\[rN((x_1,\dotsc,x_n),[\Delta_n]_{\ep'},f^r)+\{0,1,\dotsc,r-1\}\subset N((x_1,\dotsc,x_n),[\Delta_n]_{\ep},f).\]
Since $(x_1,\dotsc,x_n)\in X^n$ is an $\overline{\mathcal M}_{Q'}(1)$-adherent point of
$\Delta_n$ for $f^r$, for every $a>0$, there exists $t'>0$ such that
\[\frac 1 {t'} \#\Big(\{q'_1,\dotsc,q'_{t'}\}\cap N((x_1,\dotsc,x_n),[\Delta_n]_{\ep'},f^r)\Big)\geq 1-a.\]
Denote $t=\sum_{i=1}^{t'}\#(Q\cap I_{q'_i})$ and observe that $t'\leq t\leq r t'$ and
\[\frac 1 {t} \#\Big(\{q_1,\dotsc,q_{t}\}\cap N((x_1,\dotsc,x_n),[\Delta_n]_{\ep},f)\Big)\geq
\frac{t- a r t'}{t}\geq 1-a r.\]
Thus, $(x_1,\dotsc,x_n)$ is an $\overline{\mathcal M}_{Q}(1)$-adherent point of
$\Delta_n$ for $f$.
\end{proof}

Let $P=\{p_k\}_{k=1}^\infty$.
For every $m\geq 1$, let $J_m=(mr, (m+1)r]$.
Put $P'=\{m\in\N:\, P\cap J_m\neq\emptyset\}$ and present it as an increasing sequence $\{p'_k\}_{k=1}^\infty$.
For $\delta>0$, there exists $\delta'>0$ such that if $d(f^r(x),f^r(y))>\delta'$ then
$d(f^i(x),f^i(y))>\delta$ for $i=1,\ldots,r$.
The proof of next claim is very similar to the proof of Claim 1, therefore we leave it to the reader:

\textbf{Claim 2}: If $(x_1,\dotsc,x_n)\in X^n$ is an $\overline{\mathcal M}_{P'}(1)$-$\delta'$-escaping point of
$\Delta_n$ for $f^r$, then it is also an $\overline{\mathcal M}_{P}(1)$-$\delta$-escaping point of
$\Delta_n$ for $f$.

Now by the above two claims, if $(x_1,\dotsc,x_n)\in X^n$  is
$(\overline{\mathcal M}_{Q'}(1),\overline{\mathcal M}_{P'}(1))$-$n$-$\delta'$-scrambled for $f^r$,
then it is also $(\overline{\mathcal M}_{Q}(1),\overline{\mathcal M}_{P}(1))$-$n$-$\delta$-scrambled for $f$.
It should be noticed that the choices of $Q'$, $P'$ and $\delta'$ depend only on the system
$(X,f)$, $Q$, $P$ and $\delta$, but not on the tuple $(x_1,\dotsc,x_n)$ itself. This ends the proof.
\end{proof}

\section{Uniform Chaos}
The following two definitions come from \cite{AGHSY10} where uniform chaos was defined for the first time.
They help to define more subtle relations between points than in the case of standard scrambled sets.
Additionally, the authors in \cite{AGHSY10} proved that there are numerous dynamical systems
possessing uniformly chaotic subsets, e.g. systems with
positive topological entropy or systems chaotic in the sense of Devaney.

\begin{defn} Let $(X,f)$ be a TDS. A subset $K$ of $X$ is called
\begin{enumerate}
\item \emph{uniformly recurrent} if for every $\ep > 0$ there is an $n \in \N$ with
 $d(f^n x,x) < \ep$ for all $x$ in $K$.
\item \emph{recurrent} if every finite subset of $K$ is uniformly recurrent.
\item \emph{uniformly proximal} if for every $\ep > 0$ there is an $n \in N$ with $\textrm{diam}(f^n(K)) < \ep$.
\item \emph{proximal} if every finite subset of $K$ is uniformly proximal.
\end{enumerate}
\end{defn}

\begin{defn}
Let $(X,f)$ be a TDS. A subset $K$ of $X$ is called a \emph{uniformly chaotic set} if
there are Cantor sets $C_1\subset C_2\subset \cdots$ such that:
\begin{enumerate}
\item $K=\bigcup_{n=1}^\infty C_n$  is a recurrent subset of $X$ and also a proximal subset of $X$.
\item for each $N\in \N$, $C_N$ is uniformly recurrent.
\item for each $N\in \N$, $C_N$ is uniformly proximal.
\end{enumerate}
The system $(X,f)$ is called \emph{(densely) uniformly chaotic} if $(X,f)$ has a (dense) uniformly chaotic subset.
\end{defn}

By Theorem~\ref{thm:eta} we can easily get that if $K$ is a uniformly chaotic set then
for every $n\geq 2$ its closure contains an $(\M_Q(1),\M_Q(1))$-$n$-scrambled subset for some $Q\in\B$.
In fact, the situation is even better. We show that a uniformly chaotic set itself is
$(\M_Q(1),\M_Q(1))$-$n$-scrambled for some $Q\in\B$.

\begin{prop}
If $K\subset X$ is a uniformly chaotic set, then
there exists $Q\in\B$ such that
for every $n\geq 2$, $K$ is $(\M_Q(1),\M_Q(1))$-$n$-scrambled.
\end{prop}
\begin{proof}
By the definition of uniformly chaos, there exist Cantor sets
$C_1\subset C_2\subset \cdots\subset K$ such that $K=\bigcup_{i=1}^{\infty}C_i$ and each $C_i$
is both uniformly proximal and uniformly rigid.
For every $N\in\N$, there  are two strictly increasing sequences
$P^{(N)}=\{p^{(N)}_k\}_{k=1}^\infty$ and $S^{(N)}=\{s^{(N)}_k\}_{k=1}^\infty$ such that
for every $(x_1,\dots,x_n)\in C_N$
\[\lim_{k\to\infty} \max_{1\leq i<j\leq n} d(f^{p^{(N)}_k}(x_i), f^{p^{(N)}_k}(x_j))=0\]
and
\[\lim_{k\to\infty}(f^{s^{(N)}_k}(x_1),\cdots, f^{s^{(N)}_k}(x_n))=(x_1,\ldots,x_n).\]
By Lemma \ref{lem:SEQ}, there exists a strictly increasing sequence $Q$ such that
\[\overline{d}(P^{(N)}|Q)=\overline{d}(S^{(N)}|Q)=1, \text{ for every } N=1,2,\ldots.\]
Therefore, similarly to the proof of Theorem \ref{thm:eta}, it can be easily verified that
for every $n\geq 2$, $K$ is
$(\M_Q(1),\M_Q(1))$-$n$-scrambled.
\end{proof}

The following result may also be of some interest. It shows that uniform chaos is preserved by higher iterates of $f$ and vice-versa, so it behaves
exactly as expected from any reasonable definition of chaos.

\begin{prop}
Let $(X,f)$ be a TDS, $K\subset X$ and $r\geq 1$. Then $K$ is a uniformly chaotic set for $f$ if and only if
it is a uniformly chaotic set for $f^r$.
\end{prop}
\begin{proof}
The ``if'' part is obvious, therefore we only need to show that
if $C\subset X$ is uniformly recurrent or uniformly proximal for $f$, then it has the same property under the action of $f^r$.

If $C$ is uniformly recurrent for $f$, then there exists a sequence $\{k_i\}_{i=1}^\infty\in\B$ such that
$f^{k_i}$ converges uniformly to the identity on $C$.
Without loss of generality, we can assume that there exists $s\in\{0,1,\dotsc,r-1\}$ such that
$k_i\equiv s \pmod r$ for every $i\geq 1$.

Fix $\ep>0$ and denote $\ep_1=\frac 1 2 \ep$. There exists $n_1\in\N$ with $n_1\equiv s \pmod r$
such that $d(f^{n_1} (x),x) < \ep_1$ for all $x$ in $C$.
By the continuity of $f$, there exists $\delta_1>0$ such that
$d(f^{n_1}(y), f^{n_1}(z))<\frac{1}{2} \ep_1$, whenever $d(y,z)<\delta_1$.

Suppose that $\ep_1,\ep_2,\dotsc,\ep_m$,  $\delta_1,\delta_2,\dotsc,\delta_m$ and $n_1,n_2,\dotsc,n_m$ have been chosen
such that for every $i=1,\dotsc, m$, $\ep_{i}=\min\{\delta_{i-1}, \frac 1 2 \ep_{i-1}\}$,
$n_i\equiv s \pmod r$, $d(f^{n_i} (x),x) < \ep_i$ for all $x\in C$ and we have that
$d(f^{n_1+\dotsb+n_i}(y),f^{n_1+\dotsb+n_i}(z))<\frac 1 2 \ep_i$, provided that $d(y,z)<\delta_i$.

Let $\ep_{m+1}=\min\{\delta_{m}, \frac 1 2 \ep_{m}\}$. There exists $n_{m+1}\in\N$
with $n_{m+1}\equiv s \pmod r$ and
such that $d(f^{n_{m+1}} (x),x) < \ep_{m+1}$ for all $x$ in $C$.
When $n_{m+1}$ is fixed, using uniform continuity of $f$ we can easily find also $\delta_{m+1}$ for $\ep_{m+1}/2$ and $n_1+\ldots+n_{m+1}$
completing that way the induction.

Observe that by our construction, for every $x\in X$  we have
\begin{align*}
d(f^{n_1+\dotsb+n_{m+1}}(x),x)&\leq d(f^{n_1+\dotsb+n_{m+1}}(x),f^{n_1+\dotsb+n_{m}}(x))+
d(f^{n_1+\dotsb+n_{m}}(x),x)\\
&\leq \frac{1}{2} \ep_m+d(f^{n_1+\dotsb+n_{m}}(x),x) \leq \sum_{i=0}^{m}\frac{1}{2^i}\ep_1\\
&<\ep.
\end{align*}
In particular, we have
$d(f^{n_1+\dotsb+n_r} x,x) < \ep$ for all $x\in C$ which, since
$n_1+\dotsb+n_r\equiv 0\pmod r$, shows that indeed
$C$ is uniformly recurrent for $f^r$.

Now assume that $C$ is uniformly proximal for $f$.
For every $\ep>0$ there exists a $\delta>0$ such that $d(f^i(x),f^i(y))<\ep/2$
for $i=0,1,\ldots,r$, provided that $d(x,y)<\delta$.
Notice that for the above $\delta$ there exists an $n\in\N$ such that $\textrm{diam}(f^n(C))<\delta$.
Then $\textrm{diam}(f^{n+i}(C))<\ep$ for $i=0,1,\dotsc,r-1$. In particular there is  $j\in\{0,1,\dotsc,r-1\}$
such that $n+j\equiv 0 \pmod r$, therefore $C$ is uniformly proximal for $f^r$.
\end{proof}

\section{Interval maps}

In this section we will analyze which sequences can or cannot lead to chaotic dynamics in the context of interval maps.
From this point on $I$ will always denote the closed unit interval $[0,1]$.
By an \emph{interval map} we mean a continuous map $f\colon I\to I$ and
topological entropy of $f$ is denoted by $\htop(f)$.
It is shown in \cite{SS94} that an interval map has positive entropy
if and only if it distributionally chaotic.
In fact, we can show that for interval maps with positive entropy chaos occurs along any other sequence.
Strictly speaking we have the following.

\begin{thm}\label{thm:int_anyseq}
If $(I,f)$ satisfies $\htop(f)>0$,
then for every $n\geq 2$ there is $\delta>0$ such that for any $Q,P\subset \B$ there exists
a Mycielski $(\overline{\mathcal M}_Q(1),\overline{\mathcal M}_P(1))$-$n$-$\delta$-scrambled set.
\end{thm}
\begin{proof}
If $\htop(f)>0$, it is known (e.g. see \cite{M11}) that
there is an $f^m$-invariant closed set $\Lambda$ such that $f^m|_\Lambda$ is conjugate with
$(\Sigma_2^+,\sigma)$.
Then the result immediately follows by Theorem \ref{thm:exact} and Theorem \ref{cor:f-and-fr}.
\end{proof}

\begin{rem}
Situation in Theorem~\ref{thm:int_anyseq} is very special.
In general, there are known numerous examples of
proximal systems with positive entropy.
But such systems are never distributionally chaotic (see \cite{Pikula} or \cite{POTrans}).
In the case of these systems, some thickly syndetic sequences will be bad choices for distributional chaos.
\end{rem}

The authors in \cite{Wang07} proved that an interval map is Li-Yorke chaotic if
and only if it is distributionally chaotic in some sequence. However, as we will see,
for zero entropy interval maps there are some restrictions on the sequence.
Before stating our result, let use recall the following
periodic decomposition of $\omega$-limit sets (e.g. see \cite[Thm.~3.5]{SmitalOLS}).

\begin{lem}\label{lem:SmitalOLS}
Let $(I,f)$ be such that $\htop(f)=0$ and fix $x_0\in I$ such that $\omega(x_0, f)$ is infinite.
Then there is a sequence of (not necessarily closed) intervals
$\set{\set{L_s^i}_{i=0}^{2^s-1}}_{s=0}^\infty$ such that:
\begin{enumerate}
\item $f(L^i_s)=L^{i+1 \modn{2^s}}_s$ for $i=0,1,\ldots, 2^s-1$,
\item intervals $L_s^0,\ldots, L^{2^s-1}_s$ are pairwise disjoint for every $s\geq 0$,
\item $L^0_{s+1}\cup L^{2^s}_{s+1}\subset L^0_{s}$ for every $s\geq 0$,
\item $\omega(x_0,f^{2^s})\subset L^0_s$ and $\omega(x_0,f)\subset \bigcup_{i=0}^{2^s-1}L^i_s$ for every $s\geq 0$.
\end{enumerate}
\end{lem}

\begin{thm}\label{thm:P}
Let $(I,f)$ satisfy $\htop(f)=0$ and
assume that $P\in \B$ is such that:
\begin{enumerate}
  \item\label{dc:specP} for every $n>0$ there is $k\in\N$ such that $\# (P\cap [i-k,i+k])>n$ for every $i\in P$.
\end{enumerate}
Then no pair in $I\times I$ can be distributionally scrambled in the sequence $P$ under the action of $f$.
\end{thm}
\begin{proof}
Fix a sequence $P=\{p_t\}_{t=1}^\infty\in \B$ satisfying \eqref{dc:specP} and
suppose on the contrary that there is a pair $(x,y)\in I\times I$
which is distributionally scrambled  in the sequence $P$.
Fix a positive number $\ep$ such that the upper density of
$N((x,y), I\times I\setminus \overline{[\Delta]_\ep})$ with respect to $P$ is $1$.
Choose an integer $m$ large enough to satisfy $m \ep >1$,
and observe that $I$ can contain less than $m$ intervals with pairwise disjoint interiors and diameters at least $\ep$.
Fix a positive integer $M$  such that $7m/M<1/2$ and finally
let $k\in \N$ be such that if $i\in P$ then $\# (P\cap [i-k,i+k])>M$.

If $\omega(x,f)$ is finite then it is easy to see that  $x,y$ are either distal or asymptotic.
So the only possibility is that both $x$ and $y$ have infinite $\omega$-limit sets.
Let $\set{L_s^i}$ be the sequence of
intervals provided by Lemma~\ref{lem:SmitalOLS} for $\omega(x,f)$.

Fix $s\in\N$ such that $2^s>k$ and put $B=\set{0\leq i\leq 2^s-1:\,\diam (L^i_s)\geq \ep}$. Note that $\#(B)<m$.
The pair $(x,y)$ is proximal, so  there is $N=N(s)\geq 0$ such that $f^N(x),f^N(y)\in L^0_s$
and $f^{N+i}(x),f^{N+i}(y)\in L^{i \modn{2^s}}_s$ for every $i\geq 0$.
Therefore the set
\[ B'=\set{j\in P:\,j\geq N \text{ and } (j-N) \pmod{2^s} \in B}\]
is infinite. 

We will define sets $A_n\subset \N$ inductively.
Let $j_1\geq N$ be the first integer such that $B'\cap [j_1 k, j_1k +k)\neq \emptyset$.
Denote $A_1=[j_1k-2k,j_1k+4k)\cap P$.
Assume that sets $A_1,\dotsc, A_n$ and numbers $j_1<j_2<\dotsb < j_n$ have already been defined.
Let $j_{n+1}$ be the smallest possible integer such that $j_{n+1}>j_n+3$ and
$B'\cap [j_{n+1}k,j_{n+1}k+k)\neq \emptyset$. Denote $A_{n+1}=[j_{n+1}k-2k,j_{n+1}k+4k)\cap P$.
By the above method we generate a sequence of sets of integers $\set{A_n}_{n=1}^\infty$ such that:
\begin{enumerate}
\item if $A_i\cap A_j\neq \emptyset$ then $|i-j|\leq 1$,
\item $B'\subset \bigcup_{n=1}^\infty A_n \cup [0,N)$,
\item $\#(A_n\cap B')\leq 6m$.
\end{enumerate}

For $n\geq 1$, let $C_n=P \cap [j_n k -k , j_n k +2k)$.
Then $\#(C_n)>M$, $C_n\subset A_n$ and $\{C_i\}_{i\geq 1}$ are pairwise disjoint.
For $t>N$, let $n(t)=\max\{n\geq 1: \max A_n\leq p_t\}$. Then for sufficiently large $t$ we have
\begin{align*}
\frac 1 t\#\Big(N\big((x,y), I\times I\setminus \overline{[\Delta]_\ep}\big)\cap \{p_1,\dotsc,p_t\}\Big)&\leq
\frac 1 t \bigg(\sum_{i=1}^{n(t)} \#(A_i\cap B')+N+6k\bigg)\\
&\leq \frac{\sum_{i=1}^{n(t)} \#(A_i\cap B')}{\sum_{i=1}^{n(t)} \#(C_i)} + \frac{N+6k}{t}\\
&\leq \frac{6m}{M}+\frac{N+6k}{t} \\
&<\frac 1 2.
\end{align*}
This is a contradiction.
\end{proof}

\begin{rem}
There are many sets $P\in \B$ satisfying the assumption of Theorem~\ref{thm:P}, for example:
\begin{enumerate}
\item set with positive \emph{lower Banach density}, i.e.,
\[\liminf_{(n-m)\to\infty}\frac{P\cap\{m,m+1,\dotsc,n-m\}}{n-m+1}>0.\]
\item any ``pure'' IP-set, that is a set $P$ consisting of all sums of the form $q_{i_1}+\ldots+ q_{i_n}$,
$i_{1}<i_2<\dotsb<i_{n}$, $n>0$, where $Q=\set{q_i}\in \B$.
\end{enumerate}
\end{rem}

\section{Relations between $n$ and $(n+1)$-scrambled sets}

Recently the authors of \cite{TF11} have announced that for every $n\geq 2$,
there is a dynamical system with an uncountable distributionally $n$-scrambled set
but without any distributionally $(n+1)$-scrambled tuples.
Their example however contains an uncountable $(n+1)$-scrambled set.
Here, by a different construction,
we show that the existence of uncountable $(n+1)$-scrambled set can also be removed from the dynamics.

\begin{thm}\label{thm:stuple}
For every $n\geq 2$, there exists a subshift $X$ of $\Sigma_n$ with an uncountable distributionally $n$-scrambled set
but without any uncountable $(n+1)$-scrambled sets.
\end{thm}
\begin{proof}
Fix a sequence $\{m_k\}\in\B$ such that $m_0=0$ and $m_{k+1}-m_k>2^{k+1}$ and
a sequence $\{a_i\}\in \B$ such that $\lim_{i\ra\infty}\frac{a_i}{a_{i+1}}=0$.
Put $I_i=[a_i, a_{i+1})\cap \N$.
For any $x\in \Sigma_n^+$ define $z^x\in \Sigma_n$ by putting $z^x_{j}=x_k$ for $j\in I_{m_k}$ and
$z^x_j=0$ for $j\in \mathbb{Z}\setminus \bigcup_{k=0}^\infty I_{m_k}$.

Let $Z_0=\{z^x:\, x\in\Sigma^+_n\}$ and $Z_i=\sigma^i(Z_0)$ for $i\in\mathbb Z$.
Set
\[ X=\overline{\bigcup_{i\in\mathbb Z} Z_i}=\bigcup_{i\in\mathbb Z} Z_i \cup Q.\]
Observe that, since $\{m_k\}$ is strictly increasing, $Q$ is at most countable.
Clearly, $(X,\sigma)$ is a subshift of $(\Sigma_n,\sigma)$. We will show that $X$ satisfies our requirements.

The set $R$ of $n$-tuples $(x^{(1)},\ldots, x^{(n)})\in (\Sigma_n^+)^n$
such that for every $N$ there is $m>N$
satisfying $x^{(i)}_m\neq y^{(i)}_m$ for $1\leq i<j\leq n$ is residual in $\Sigma_n^+$.
But then by Mycielski theorem there is a Cantor set $C\subset \Sigma_n^+$
such that for every $n$-tuples $(x^{(1)},\ldots, x^{(n)})\in C^n\setminus \Delta^{(n)}$,
there is an increasing sequence $\{q_k\}$ such that $x^{(i)}_{q_k}\neq y^{(j)}_{q_k}$ for $1\leq i<j\leq n$.
Now, let $S=\set{z^x:\, x\in C}$. If we fix any tuple $(z^1,\ldots, z^n)\in S^n \setminus \Delta^{(n)}$
then by the choice of $C$ there is an increasing sequence $\{q_k\}$ such that
$z^i_t \neq z^j_t$ for every $t\in I_{m_{q_k}}$ and $1\leq i<j\leq n$.
Similarly, there is an increasing sequence $\{q_k\}$ such that
$z^i_t = z^j_t=0$ for every $t\in I_{m_{q_k}}$ and $1\leq i<j\leq n$.
By the choice of sequence $\{a_i\}$ we immediately get that
the $n$-tuple $(z^1,\ldots, z^n)$ is distributionally $n$-scrambled,
and so $S$ is an uncountable distributionally $n$-scrambled set.

Next, we show that there is no uncountable $(n+1)$-scrambled set in $X$.
On the contrary, assume that $S\subset X$ is an uncountable $(n+1)$-scrambled set.
Since $Q$ is at most countable, there exists an $i\in\mathbb{Z}$ such that $S\cap Z_i$ is uncountable.
Without loss of generality, assume that $S\subset Z_0$.
This is a contradiction. To finish the proof, it is enough to show that $Z_0$  does not contain $(n+1)$-scrambled tuples.

Let $z^1,z^2,\dotsc,z^{n+1} \in Z_0$.
Fix any $\ep>0$ and $N>0$ such that $2^{-N}<\ep$.
For any $k>m_{3N}$,  by the definition of $Z_0$ we have that
each $z^i_{[k-N,k+N]}$ is a block of the form $0^r a^s 0^t$ for some $r,s,t\geq 0$ with $r+s+t=2N+1$ and
some symbol $a\in\{0,1,\dots,n-1\}$.
By the pigeonhole principle, there exist $i_1\neq i_2$ such that
$z^{i_1}_{[k-N,k+N]}=z^{i_2}_{[k-N,k+N]}$. Therefore,
\[ \min_{1\leq i<j\leq n+1}d(\sigma^k(z^i),\sigma^k(z^j))<\ep,\]
and $(z^1,z^2,\dotsc, z^{n+1})$ cannot be an $(n+1)$-scrambled tuple.
\end{proof}

Note that the subshift constructed in the proof of Theorem~\ref{thm:stuple} contains $(n+1)$-scrambled tuples.
Then a natural question is as follows:

\begin{que}
Is there a system $(X,f)$ with an uncountable distributionally $2$-scrambled set but without any $3$-scrambled tuples?
\end{que}

\subsection*{Acknowledgement}
The first author was  supported by the National Natural Science Foundation of China,
grants no.~11071231, 11001071 and 11171320.
The research of second author was supported by the Polish Ministry of Science and Higher Education (2012). His research leading to results contained in this paper was supported by the Marie Curie
European Reintegration Grant of the European Commission under grant agreement no.~PERG08-GA-2010-272297.

\end{document}